\documentclass[12pt,twoside]{article}
\usepackage{amssymb,amsmath,amsthm,latexsym}
\usepackage{amsfonts}
\usepackage{graphicx}
\newtheorem{thm}{Theorem}[section]

\newtheorem{lem} [thm]{Lemma}
\theoremstyle{definition}	

\newtheorem{defn}[thm]{Definition}
\numberwithin{equation}{section}

\newcommand\diag{\operatorname{diag}}

\begin{document}
	\label{'ubf'}
	\setcounter{page}{1} 
	\markboth {\hspace*{-9mm} \centerline{\footnotesize \sc
		Signed Distance Laplacian Matrices for Signed Graphs}
	}
	{ \centerline {\footnotesize \sc
				Roshni, Germina, Hameed, Zaslavsky
		} \hspace*{-9mm}
	}

	\begin{center}
		{
			{\huge \textbf{Signed Distance Laplacian Matrices for Signed Graphs
				}
			}
			\\
			
			\medskip

			Roshni T Roy\footnote{\small Department of Mathematics, Central University of Kerala, Kasaragod - 671316,\ Kerala,\ India.\ Email:roshnitroy@gmail.com}
			Germina K A\footnote{\small  Department of Mathematics, Central University of Kerala, Kasaragod - 671316,\ Kerala,\ India.\ Email: srgerminaka@gmail.com}
			Shahul Hameed K\footnote{\small  Department of Mathematics, K M M Government\ Women's\ College, Kannur - 670004,\ Kerala,  \ India.  E-mail: shabrennen@gmail.com}
			Thomas Zaslavsky\footnote{\small Department of Mathematical Sciences, Binghamton University (SUNY), Binghamton, NY 13902-6000, U.S.A.  E-mail: zaslav@math.binghamton.edu}
			
		}
	\end{center}
	\thispagestyle{empty}
	\begin{abstract}
		A signed graph is a graph whose edges are labeled either positive or negative. Corresponding to the two signed distance matrices defined for signed graphs, we define two signed distance laplacian matrices. We characterize balance in signed graphs using these matrices and find signed distance laplacian spectra of some classes of unbalanced signed graphs.
	\end{abstract}

	

	\section{Introduction}
	
	Throughout this article, unless otherwise mentioned, by a graph we mean a finite, connected, simple graph.  For any terms which are not mentioned here, the reader may refer to \cite{fh}.
	
	A signed graph $\Sigma=(G,\sigma)$ is an underlying graph $G=(V,E)$ with a signature function $\sigma:E\rightarrow \{1,-1\}$. Two types of signed distances matrices are introduced by Hameed et al.\ \cite{sh}. In this paper, we define two signed distance laplacian matrices for signed graphs and characterize balance in signed graphs using these matrices. The signed distance laplacian spectra of some classes of unbalanced signed graphs are also studied.
	
	First we recall the definition of signed distances and corresponding signed distance matrices defined in \cite{sh}.
	Given a signed graph $\Sigma=(G,\sigma)$, the sign of a path $P$ in $\Sigma$ is defined as $\sigma(P)=\prod_{e\in E(P)} \sigma(e)$. The shortest path between two given vertices $u$ and $v$ is denoted by $P_{(u,v)}$ and the collection of all shortest paths $P_{(u,v)}$ by $\mathcal{P}_{(u,v)}$; and $d(u,v)$ denotes the usual distance between $u$ and $v$.
	
	\begin{defn}[Signed distance matrices \cite{sh}]
		Auxiliary signs are defined as:
		\par(S1) $\sigma_{\max}(u,v) = -1$ if all shortest $uv$-paths are negative, and $+1$ otherwise.
		\par(S2) $\sigma_{\min}(u,v) = +1$ if all shortest $uv$-paths are positive, and $-1$ otherwise.
		
		Signed distances are:
		\par(d1) $d_{\max}(u,v) = \sigma_{\max}(u,v) d(u,v)=\max\{\sigma(P_{(u,v)}): P_{(u,v)} \in \mathcal{P}_{(u,v)} \}d(u,v).$
		\par(d2) $d_{\min}(u,v) = \sigma_{\min}(u,v) d(u,v)=\min\{\sigma(P_{(u,v)}): P_{(u,v)} \in \mathcal{P}_{(u,v)} \}d(u,v).$
		
		And the signed distance matrices are:
		\par(D1)  $D^{\max}(\Sigma)=(d_{\max}(u,v))_{n\times n}$.
		\par(D2)  $D^{\min}(\Sigma)=(d_{\min}(u,v))_{n\times n}$.
	\end{defn}
	
	\begin{defn}[\cite{sh}]
		Two vertices $u$ and $v$ in a signed graph $\Sigma$ are said to be \emph{distance-compatible} (briefly, \emph{compatible}) if $d_{min}{(u,v)}=d_{max}{(u,v)}$.  And $\Sigma$ is said to be (distance-)compatible if every two vertices are compatible. Then $D^{\max}(\Sigma)= D^{\min}(\Sigma)= D^{\pm}(\Sigma).$
	\end{defn}
The two complete signed graphs from the distance matrices $D^{\max}$ and $D^{\min}$ is defined as follows.
	\begin{defn}[\cite{sh}]
		The associated signed complete graph $K^{D^{\max}}(\Sigma)$ with respect to $D^{\max}(\Sigma)$ is obtained by  joining the non-adjacent vertices of $\Sigma$ with edges having signs
		\begin{equation*}
		\sigma(uv)= \sigma_{\max}(uv).
		\end{equation*}
		
		The associated signed complete graph $K^{D^{\min}}(\Sigma)$ with respect to $D^{\min}(\Sigma)$ is obtained by joining the non-adjacent vertices of $\Sigma$ with edges having signs
		\begin{equation*}
		\sigma(uv)= \sigma_{\min}(uv).
		\end{equation*}
	\end{defn}
	
	The distance laplacian matrix of a graph was introduced and studied by Mustapha et al.\ in \cite{ma}. We introduce two signed distance laplacian matrices for signed graphs as follows.
	
	The transmission $Tr(v)$ of a vertex $v$ is defined to be the sum of the distances from $v$ to all other vertices in $G$. That is, $Tr(v) = \sum_{u \in V(G)} d(v,u)$.
	The transmission matrix $Tr(G)$ for a graph is the diagonal matrix with diagonal entries $Tr(v_i)$.
	
	\begin{defn}	
	\rm{Now we define two signed distance laplacian matrices for signed graphs as
		\par(L1)$L^{\max}(\Sigma) = Tr(G)- D^{\max}(\Sigma).$
		\par(L2)$L^{\min}(\Sigma) = Tr(G)- D^{\min}(\Sigma).$
		
	When $\Sigma$ is compatible, $L^{\max}(\Sigma) = L^{\min}(\Sigma) = L^{\pm}(\Sigma)$.}
    \end{defn}

\section{Balanced signed graphs}
To study balance in signed graphs using signed distance laplacian matrices, it is necessary to study the laplacian matrix of weighted signed graphs. Let $\Sigma =(G,\sigma)$ be a signed graph and let $w$ be a positive weight function defined on the edges of $\Sigma$. We denote a weighted signed graph by $(\Sigma,w)$. For a weighted signed graph $(\Sigma,w)$, $w(\Sigma)$ is the product of all the weights given to the edges of $\Sigma$. We use the notation $u\sim v$ when the vertices $u$ and $v$ are adjacent and similar notation for the incidence of an edge on a vertex.

\begin{defn}
	\rm{Let $(\Sigma,w)$ be a weighted signed graph. Its adjacency matrix $A(\Sigma, w)=(a_{ij})_n$ is defined as the square matrix of order $n=|V(G)|$ where
		
		$a_{ij} =
		\left\{
		\begin{array}{ll}
		\sigma(v_iv_j)w(v_iv_j)  & \mbox{if } v_i\sim v_j \\
		0 & \mbox{otherwise. }
		\end{array}
		\right.$ }
\end{defn}

\begin{defn}
	\rm{For a weighted signed graph $(\Sigma,w)$, its weighted laplacian matrix is defined as $L(\Sigma,w)=D(\Sigma,w) - A(\Sigma,w)$ where the diagonal matrix $D(\Sigma,w)$ is $\diag\big(\sum_{e:v_i\sim e}w(e)\big)$.}
	The matrix $D(\Sigma,w)$ is the weighted degree matrix of $(\Sigma,w).$
\end{defn}
To discuss oriented incidence matrix, we orient the edges (in an arbitrarily but fixed way) of the weighted signed graph.
For an oriented edge $\vec{e}_j=\overrightarrow{v_iv_k}$ we take $v_i$ as the tail of that edge and $v_k$ as its head and we write $t(\vec{e}_j)=v_i$ and $h(\vec{e}_j)=v_k$. We choose weights $w$ from the set of positive real numbers and take $\sqrt{w}$ as the positive square root of $w$.
\begin{defn}
	Given a weighted signed graph $(\Sigma, w)$, its (oriented) weighted incidence matrix is defined as $\mathrm{H}(\Sigma,w)=(\eta_{v_i{e}_j})$ where
	$$\eta_{v_i{e}_j} =
	\begin{cases}
	\sigma(e_j)\sqrt{w(e_j) }& \mbox{if } t(\vec{e}_j)=v_i ,\\
	-\sqrt{w(e_j) }& \mbox{if } h(\vec{e}_j)=v_i ,\\
	0 & \mbox{otherwise. }
	\end{cases}
	$$
\end{defn}

Let $\mathrm{H}^T(\Sigma,w) = (\eta'_{{e}_{i}v_j})$ be the transpose of the weighted incidence matrix $\mathrm{H}(\Sigma,w)$.  Thus, $\eta'_{{e}_{i}v_j} = \eta_{v_j{e}_{i}}$.

\begin{thm}
	For a weighted signed graph $(\Sigma,w)$, $L(\Sigma,w) =\mathrm{H}(\Sigma,w)\mathrm{H}^T(\Sigma,w).$
\end{thm}
\begin{proof}
	
	Let $v_1,v_2, \dots , v_n$ and $\vec{e}_1,\vec{e}_{2}, \dots ,\vec{e}_{m}$ be the vertices and edges in $\Sigma$, respectively.  The $(i,j)^{th}$ entry of $\mathrm{H}\mathrm{H}^T$ is $ \sum_{k=1}^{m}\eta_{v_i{e}_{k}}\eta'_{{e}_{k}v_j}$.

	For $i=j$, $\eta_{v_i{e}_{k}}\eta'_{{e}_{k}v_j} = \eta_{v_i{e}_{k}}^2 \neq 0$ if and only if ${e}_{k}$ is incident to $v_i$.  Then $\eta_{v_i{e}_{k}}\eta'_{{e}_{k}v_j} = (\pm\sqrt{w(e_k)})^2 = w(e_k)$. Thus, the diagonal entry in $\mathrm{H}\mathrm{H}^T$ is $\sum_{{e}:v_i\sim{e}}w(e).$
	
	For $i \neq j$, $\eta_{v_ie_{k}}\eta'_{{e}_{k}v_j}\neq 0$ if and only if ${e}_{k}$ is an edge joining $v_i$ and $v_j$.  Then $\eta_{v_ie_{k}}\eta'_{{e}_{k}v_j} = -\sigma(e_k)w(e_k)$.
	
	In both cases, the $(i,j)^{th}$ entry of $\mathrm{H}\mathrm{H}^T$ coincides with the $(i,j)^{th}$ entry of $L(\Sigma,w)$, hence the proof.
\end{proof}

\begin{lem}\label{tree}
	For a weighted signed tree $(\Sigma,w)$, $\det L(\Sigma,w)= 0.$
\end{lem}
\begin{proof}
	A tree on $n$ vertices has $n-1$ edges. Thus, $\mathrm{H}(\Sigma,w)$ is a matrix of order $n \times (n-1)$ and hence $L(\Sigma,w) = \mathrm{H}(\Sigma,w)\mathrm{H}^T(\Sigma,w)$ has rank less than $n$. This implies $\det L(\Sigma,w)= 0.$
\end{proof}

\begin{lem}\label{cycle}
	Let $(\Sigma,w)$ be a weighted signed graph where the underlying graph is a cycle $C_n$ of order $n$. Then $\det L(\Sigma,w)= 2w(C_n)(1-\sigma(C_n))$.
\end{lem}
\begin{proof}
	Let the cycle be $C_n = v_1\vec{e}_{1}v_2\vec{e}_{2}v_3\vec{e}_{3} \cdots v_{n-1}\vec{e}_{n-1}v_n\vec{e}_{n}v_1$.  The weighted incidence matrix $\mathrm{H}(\Sigma,w)$ is
	$$ \begin{pmatrix}
	\sigma(e_1)\sqrt{w(e_1)}& 0                   & \dots  & 0    & -\sqrt{w(e_n)}\\
	-\sqrt{w(e_1)}          & \sigma(e_2)\sqrt{w(e_2)} & \dots  &  0   & 0\\
	\vdots              &\vdots 				 & \ddots  &\vdots &\vdots\\
	0 				   & 0					 & \dots  & \sigma(e_{n-1})\sqrt{w(e_{n-1})} & 0 \\
	0                  & 0                   & \dots  & -\sqrt{w(e_{n-1})} & \sigma(e_n)\sqrt{w(e_n)}
	\end{pmatrix}.$$
Expanding along the first row to find the determinant we get
	$$ \det \mathrm{H}(\Sigma,w) = \sigma(e_1)\sqrt{w(e_1)}M_{1,1} + (-1)^n\sqrt{w(e_n)} M_{1,n}$$ 
where
	$$ M_{1,1} = \det \begin{pmatrix}
	\sigma(e_2)\sqrt{w(e_2)}  & \dots  &  0            & 0\\
	\vdots 				 & \ddots  &\vdots          &\vdots\\
    0			         & \dots  & \sigma(e_{n-1})\sqrt{w(e_{n-1})} & 0 \\
	 0                   & \dots  & -\sqrt{w(e_{n-1})} & \sigma(e_n)\sqrt{w(e_n)}
	\end{pmatrix}, $$
	
	$$ M_{1,n} = \det \begin{pmatrix}
	-\sqrt{w(e_1)}  &\sigma(e_2)\sqrt{w(e_2)} & \dots  &  0 \\
	\vdots  	   & \ddots  &\vdots   &\vdots\\
	0		   & 0      & \dots  & \sigma(e_{n-1})\sqrt{w(e_{n-1})} \\
	0          & 0      & \dots  & -\sqrt{w(e_{n-1})}
	\end{pmatrix}. $$
Since $M_{1,1}$ and $M_{1,n}$ are determinants of triangular matrices, 
	\begin{align*}  \det \mathrm{H}(\Sigma,w) &= \sigma(e_1)\sqrt{w(e_1)}\sigma(e_2)\sqrt{w(e_2)} \cdots \sigma(e_n)\sqrt{w(e_n)} \\
	& \qquad + (-1)^n\sqrt{w(e_n)}(-\sqrt{w(e_1)}) \cdots (-\sqrt{w(e_{n-1})})\\
	& = (\sigma(C_n) - 1)\sqrt{w(C_n)}.
	\end{align*}
Now, $\det \mathrm{H}^T(\Sigma,w) = \det \mathrm{H}(\Sigma,w),$ hence, 
	\begin{align*}
	\det L(\Sigma,w) & = (\sigma(C_n) - 1)\sqrt{w(C_n)}.(\sigma(C_n) - 1)\sqrt{w(C_n)}\\
	                  & = 2w(C_n)(1-\sigma(C_n)).
	\qedhere
	\end{align*}
\end{proof}

\begin{lem}\label{unicycle}
	Let $(\Sigma,w)$ be a signed graph, where the underlying graph is a unicyclic graph of order $n$ with unique cycle $C$. Then $\det L(\Sigma,w) = 2w(\Sigma)(1-\sigma(C)).$
\end{lem}

\begin{proof}
	Define the orientation of edges so that for $i<j$ the edge $\vec{e}_{i,j}$ has tail $v_i$ and head $v_j$. Let $C = v_1\vec{e}_{1}v_2\vec{e}_{2} \cdots v_p\vec{e}_{p}v_1$ be the unique cycle, and label the vertices so that each edge $\vec{e}_{ij}$ not in $C$ has $v_i$ nearer to $C$ than $v_j$; in other words, the vertex labels increase when moving away from $C$.  
	Then the incidence matrix $\mathrm{H}(\Sigma,w)$ has the following form:
	$$\left(\begin{array}{cc|c}
	\mathrm{H}(C,w)	&&{*}	\\
	\hline
	O		&&\begin{matrix}-\sqrt{w(e_{p+1}}) &*&\dots	&*&*\\
			0	&-\sqrt{w(e_{p+2}})	&\dots	&*	&*\\
			0	&0		&\dots	&*	&*\\
			\vdots	&\vdots		&\ddots	&\vdots	&\vdots\\
			0	&0		&\dots&0	&-\sqrt{w(e_{n}})
			\end{matrix}
	\end{array}\right),$$
which is an upper-triangular block matrix whose first diagonal block is $\mathrm{H}(C,w)$ and whose other diagonal elements correspond to the heads of edges not in $C$. 
	Hence,
	\begin{align*}
	  \det \mathrm{H}(\Sigma,w) &= \det \mathrm{H}(C,w) \cdot \prod_{k=p+1}^{n}(-\sqrt{w(e)})\\
				  &= (\sigma(C) - 1) (-1)^{n-p}\sqrt{w(\Sigma)}.
	\end{align*}	
	Thus,
	 \begin{align*}
	       \det L(\Sigma,w) & = \det \mathrm{H}(\Sigma,w)\det \mathrm{H}^T(\Sigma,w)\\
	       			 &= 2w(\Sigma)(1-\sigma(C)).
	 \qedhere      					
	 \end{align*}				
\end{proof}

\begin{lem}\label{forest}
	Let $(\Sigma,w)$ be a signed graph whose underlying graph is a $1$-forest.  Then $\det L(\Sigma,w)=  w(\Sigma) \prod_{\psi} 2(1-\sigma(C_\psi))$, where the product runs over all component $1$-trees $\psi$ having unique cycle $C_\psi.$
\end{lem}
\begin{proof}
	By suitable reordering of vertices and edges we can make the laplacian $L(\Sigma,w)$ into a block diagonal matrix, where the blocks correspond to the components of the $1$-forest;.  Thus, $\det L(\Sigma,w)$ is the product of the laplacian determinants of the components. The components are $1$-trees. By using Lemma \ref{unicycle} we get the expression for the determinant.
\end{proof}

A signed graph is said to be contrabalanced if it contains no positive cycles \cite{tz}.
\begin{lem}\label{lemma}
	Let $L(\Sigma,w)$ be a weighted signed graph having $n$ vertices and $\Psi$ be a spanning subgraph of $(\Sigma,w)$ having exactly $n$ edges. Then, $\det{L(\Psi,w)} \neq 0$ if and only if $\Psi$ is a contrabalanced $1$-forest.
\end{lem}
\begin{proof}
	A spanning subgraph $\Psi$ of $n$ edges that is not a $1$-forest must have a component that has fewer edges than vertices, that is, it is a tree $T$, and $\det L(T,w) = 0$ by Lemma \ref{tree}.  Since $\det{L(\Psi,w)}$ is the product of the laplacians of the components of $(\Psi,w)$, $\det{L(\Psi,w)} = 0$ if $\Psi$ is not a $1$-forest.
	
	Now, assuming $(\Psi,w)$ is a $1$-forest, by Lemma \ref{forest}, $\det{L(\Psi,w)} \neq 0$ if and only if $\Psi$ contains no positive cycles, which means that $\Psi$ is a contrabalanced $1$-forest.
\end{proof}

Let $c(\Psi)$ denote the number of components of a signed graph $\Psi$.

\begin{thm}\label{gen}
	Let $(\Sigma,w)$ be a signed graph; then 
	$$\det L(\Sigma,w)=  \sum_{\Psi} 4^{c(\Psi)}w(\Psi),$$  
where the summation runs over all contrabalanced spanning $1$-forests $\Psi$ of $G$.
\end{thm}
\begin{proof}
	Since $L(\Sigma,w)=\mathrm{H}(\Sigma,w)\mathrm{H}^T(\Sigma,w)$, by the Binet--Cauchy theorem \cite{jgb} we get
	$$\det L(\Sigma,w) = \sum_J \det \mathrm{H}(J) \det\mathrm{H^T}(J) =  \sum_J\det{L(J,w)},$$
	where $J$ is a spanning subgraph of $G$ with exactly $n$ edges.
	
	Thus by Lemma \ref{lemma}, $\det L(\Sigma,w) =  \sum_{\Psi}\det{L({\Psi})}$ where the summation runs over all contrabalanced spanning $1$-forests $\Psi$ of $G.$  Since every cycle $C_\psi$ in $\Psi$ is negative, the factor $1-\sigma(C_\psi) = 2$.  That gives the formula of the theorem.
\end{proof}

Recall that we assume the signed graph $\Sigma$ is connected.

\begin{thm}\label{detbal}
	A weighted signed graph $(\Sigma,w)$ is balanced if and only if the determinant of its weighted laplacian matrix is equal to $0.$
\end{thm}
\begin{proof}
	In the formula of Theorem \ref{gen} every term is positive, so the determinant is $0$ if and only if there is no contrabalanced $1$-forest in $\Sigma$.  Since $\Sigma$ is connected, that implies it has no negative circle, that is, it is balanced.
\end{proof}

We recall the characterization theorem for balance in signed graphs using signed distances proved by Hameed et al.\ in \cite{sh}.
\begin{thm}[\cite{sh}]\label{dist}
	For a signed graph $\Sigma$ the following statements are equivalent:
	\begin{enumerate}
		\item [\rm{(i)}] $\Sigma$ is balanced.
		\item [\rm{(ii)}] The associated signed complete graph $K^{D^{\max}}(\Sigma)$ is balanced.
		\item [\rm{(iii)}] The associated signed complete graph $K^{D^{\min}}(\Sigma)$ is balanced.
		\item [\rm{(iv)}] $D^{\max}(\Sigma)=D^{\min}(\Sigma)$ and the associated signed complete graph $K^{D^{\pm}}(\Sigma)$ is balanced.
	\end{enumerate}	
\end{thm}

Now we are ready to characterize balance in signed graphs using signed distance laplacian matrices.
\begin{thm}\label{main}
	The following properties of a signed graph $\Sigma$ are equivalent.
	\begin{enumerate}
		\item [\rm{(i)}] $\Sigma$ is balanced.
		\item [\rm{(ii)}] The max-signed distance laplacian determinant $\det L^{\max}(\Sigma) = 0.$
		\item [\rm{(iii)}] The min-signed distance laplacian determinant $\det L^{\min}(\Sigma) = 0.$
		\item [\rm{(iv)}] $L^{\max}(\Sigma) = L^{\min}(\Sigma)$ and $\det L^{\pm}(\Sigma) = 0.$
	\end{enumerate}	
\end{thm}
\begin{proof}
	Corresponding to the associated signed complete graph $K^{D^{\max}}(\Sigma)$, we define a weighted signed complete graph $(K^{D^{\max}}(\Sigma), w)$ where $w(e)= d_{\max}(u,v)$ and $(K^{D^{\min}}(\Sigma), w)$, where $w(e)= d_{\min}(u,v)$ for an edge $e=uv$. Then
	\begin{center}
		$L(K^{D^{\max}}(\Sigma), w)=L^{\max}(\Sigma)$ and
		$L(K^{D^{\min}}(\Sigma), w)=L^{\min}(\Sigma)$.
	\end{center}
    Thus, by Theorem \ref{detbal}, 
\begin{enumerate}
\item[]    $\det L^{\max}(\Sigma)= 0$ if and only if $K^{D^{\max}}(\Sigma)$ is balanced, and
\item[]    $\det L^{\min}(\Sigma) = 0$ if and only if $K^{D^{\min}}(\Sigma)$ is balanced.
\end{enumerate}	
	Hence, by Theorem \ref{dist} we get the required characterization.
\end{proof}

\section{Signed distance laplacian spectrum}
Now, we move to the signed distance laplacian spectrum of some types of signed graph.
\begin{thm}
	A signed graph $\Sigma$ is balanced if and only if $L^{\max}(\Sigma) = L^{\min}(\Sigma) = L^{\pm}(\Sigma)$ and $L^{\pm}(\Sigma)$ is cospectral with $L(G)$.
\end{thm}
\begin{proof}
	Suppose $\Sigma=(G,\sigma)$ is balanced. Then $\Sigma$ can be switched to an all positive signed graph $\Sigma^\zeta = (G,+)$. Now, $D^{\pm}(\Sigma)$ exists \cite{sh} and is similar to $D^{\pm}(\Sigma^\zeta)$, that is, $D^{\pm}(\Sigma^\zeta) = SD^{\pm}(\Sigma)S^{-1}$ \cite{sh}.
	So,
	\begin{align*}
	SL^{\pm}(\Sigma)S^{-1} &=S(Tr(G)- D^{\pm}(\Sigma))S^{-1}
	=Tr(G)- D^{\pm}(\Sigma^\zeta)\\
	&=L^{\pm}(\Sigma^\zeta).
	\end{align*}
	Thus, $L^{\pm}(\Sigma)$ is similiar to $L(G)$, which implies that $L^{\pm}(\Sigma)$ is cospectral with $L(G)$.
	
	Conversely, suppose $L^{\max}(\Sigma) = L^{\min}(\Sigma) = L^{\pm}(\Sigma)$ and $L^{\pm}$ is cospectral with $L(G)$. Thus, $\det L^{\pm}(\Sigma) = \det L(G) = 0$ and hence, by Theorem \ref{main}, $\Sigma$ is balanced.
\end{proof}

A signed graph $\Sigma$ is $t$-transmission regular if  $Tr(v) = \sum_{u \in V(G)} d(v,u)=t $ for all $v \in V(G).$ Odd cycles $C_{2k+1}$ are $k(k+1)$-transmission regular and even cycles $C_{2k}$ are $k^2$-transmission regular.

\begin{thm}
If the signed graph $\Sigma$ is $t$-transmission regular, then the signed distance laplacian eigenvalues of $L^{\max}(\Sigma)(or\  L^{\min}(\Sigma))$ are $t-\lambda$, where $\lambda$ is an eigenvalue of $D^{\max}(\Sigma)(or \ D^{\min}(\Sigma))$.
\end{thm}

For the odd unbalanced cycle $C^{-}_{n}$, where $n=2k+1$, there is a unique shortest path between any two vertices. Thus, $L^{\max}(C^{-}_{n})=L^{\min}(C^{-}_{n})=L^{\pm}(C^{-}_{n})$. The signed distance spectrum of an odd unbalanced cycle is given in \cite{sh}. Thus, we get, the signed distance laplacian spectrum of $C^{-}_{n}$ as an immidiate corollary.
\begin{thm}
	For an odd unbalanced cycle $C^{-}_{n}$, where $n=2k+1$, the spectrum of $L^\pm$ is
	$$\begin{pmatrix}
	k(k+1)-k(-1)^k-\dfrac{1-(-1)^k}{2}  & k(k+1)-\dfrac{k(-1)^j}{\sin((2j+1)\frac{\pi}{2n})}-\dfrac{\sin^2((2j+1)\frac{k\pi}{2n})}{\sin^2((2j+1)\frac{\pi}{2n})} \\
	1 & 2 \quad (j=0,1,2,\ldots, k-1)
\end{pmatrix}.$$
\end{thm}

\section*{Acknowledgement}
	The first author would like to acknowledge her gratitude to Department of Science and Technology, Govt. of India for the financial support under INSPIRE Fellowship scheme Reg No: IF180462. The second author would like to acknowledge  her gratitude to Science and Engineering Research Board (SERB), Govt. of India, for the financial support under the scheme Mathematical Research Impact Centric Support (MATRICS), vide order no.: File No. MTR/2017/000689.


\section*{References}
   \begin{enumerate}
   	
   	\bibitem{ma} M. Aouchiche, P. Hansen, Two Laplacians for the distance matrix of a graph, Linear Algebra Appl.\ 439 (2013) 21--33.
   	\bibitem{jgb} J. G. Broida, S. G. Williamson, {\it Comprehensive Introduction to Linear Algebra},  Addison Wesley, Redwood City, Cal., 1989.
	\bibitem{sh} S.\ Hameed K, Shijin T V, Soorya P, Germina K A, T. Zaslavsky, Signed distance in signed graphs, Linear Algebra Appl.\ 608 (2021) 236--247.
	\bibitem{fh} F. Harary, {\it Graph Theory},  Addison Wesley, Reading, Mass., 1969.
	\bibitem{tz} T. Zaslavsky, Biased graphs. I. Bias, balance, and gains, J. Combinatorial
	Theory Ser.\ B 47 (1989) 32--52.
	
	\end{enumerate}
\end{document}